\documentclass[12pt]{amsart}

\usepackage[mathscr]{eucal}
\usepackage{amsmath,amssymb,amscd,amsthm}
\usepackage{color}
\usepackage{epsfig}
\newtheorem{theorem}{Theorem}

\newtheorem{definition}{Definition}
\newtheorem{lemma}{Lemma}

\theoremstyle{definition}

\theoremstyle{plane}

\def \beq{ \begin{equation} }
\def \eeq{\end{equation}}

\title{Homographic solutions of the curved 3-body problem}
\begin{document}
\maketitle
\markboth{Florin Diacu and Ernesto P\'erez-Chavela}{Homographic solutions of the curved $3$-body problem}
\author{\begin{center}
Florin Diacu\\
\smallskip
{\footnotesize Pacific Institute for the Mathematical Sciences\\
and\\
Department of Mathematics and Statistics\\
University of Victoria\\
P.O.~Box 3060 STN CSC\\
Victoria, BC, Canada, V8W 3R4\\
diacu@math.uvic.ca\\
}\end{center}

\begin{center}
and
\end{center}

\begin{center}
Ernesto P\'erez-Chavela\\
\smallskip
{\footnotesize
Departamento de Matem\'aticas\\
Universidad Aut\'onoma Metropolitana-Iztapalapa\\
Apdo.\ 55534, M\'exico, D.F., M\'exico\\
epc@xanum.uam.mx\\
}\end{center}

}

\vskip0.5cm

\begin{center}
\today
\end{center}

\begin{abstract}
In the 2-dimensional curved $3$-body problem, we prove the existence of Lagrangian 
and Eulerian homographic orbits, and provide their complete classification in the
case of equal masses. We also show that the only non-homothetic hyperbolic Eulerian solutions are the hyperbolic Eulerian relative equilibria, a result that proves their instability.

\end{abstract}

\section{Introduction}

We consider the $3$-body problem in spaces of constant curvature ($\kappa\ne 0$), 
which we will call {\it the curved $3$-body problem}, to distinguish it from its classical 
Euclidean ($\kappa=0$) analogue. The study of this problem might help us understand the nature of the physical space. Gauss allegedly tried to determine the nature of space by measuring the angles of a triangle formed by the peaks of three mountains. Even if the goal of his topographic measurements was different from what anecdotical history attributes to him (see \cite{Mill}), this method of deciding the nature of space remains valid for 
astronomical distances. But since we cannot measure the angles of cosmic triangles, we could alternatively check whether specific (potentially observable) motions of celestial bodies occur in spaces of negative, zero, or positive curvature, respectively.

In \cite{Diacu1}, we showed that while Lagrangian orbits (rotating equilateral triangles having the bodies at their vertices) of non-equal masses are known to occur for $\kappa=0$, they must have equal masses for $\kappa\ne 0$. Since Lagrangian solutions 
of non-equal masses exist in our solar system (for example, the triangle formed by the Sun, Jupiter, and the Trojan asteroids), we can conclude that, if assumed to have constant curvature, the physical space is Euclidean for distances of the order $10^1$ AU. The discovery of new orbits of the curved $3$-body problem, as defined here in the spirit of an old tradition, might help us extend our understanding of space to larger scales.

This tradition started in the 1830s, when Bolyai and Lobachevsky proposed a  {\it curved 2-body problem}, which was broadly studied (see most of the 77 references in \cite{Diacu1}). But until recently nobody extended the problem beyond two bodies. The newest results occur in \cite{Diacu1}, a paper in which we obtained a unified framework that offers the equations of motion of the {\it curved $n$-body problem} for any $n\ge 2$ and $\kappa\ne 0$. We also proved the existence of several classes of {\it relative equilibria}, including the Lagrangian orbits mentioned above. Relative equilibria are orbits for which the configuration of the system remains congruent with itself for all time, i.e.\ the distances between any two bodies are constant during the motion. 

So far, the only other existing paper on the curved $n$-body problem, treated in a unified context, deals with singularities, \cite{Diacu1bis}, a subject we will not approach here. But relative equilibria can be put in a broader perspective. They are also the object of Saari's conjecture (see \cite{Saari}, \cite{Diacu2}), which we partially solved for the curved $n$-body problem, \cite{Diacu1}. Saari's conjecture has recently generated a lot of interest in classical celestial mechanics (see the references in \cite{Diacu2}, \cite{Diacu3}) and is still unsolved for $n>3$. Moreover, it led to the formulation of Saari's homographic conjecture, \cite{Saari}, \cite{Diacu3}, a problem that is directly related to the purpose of this research.

We study here certain solutions that are more general than relative equilibria, namely orbits for which the configuration of the system remains similar with itself. In this class of solutions, the relative distances between particles may change proportionally during the motion, i.e.\ the size of the system could vary, though its shape remains the same. We will call these solutions {\it homographic}, in agreement with the classical terminology, \cite{Win}.

In the classical Newtonian case, \cite{Win}, as well as in more general classical contexts, \cite{Diacu0}, the standard concept for understanding homographic solutions is that of {\it central configuration}. This notion, however, seems to have no meaningful analogue in
spaces of constant curvature, therefore we had to come up with a new approach.

Unlike in Euclidean space, homographic orbits are not planar, unless they are relative equilibria. In the case $\kappa>0$, for instance, the intersection between a plane and a sphere is a circle, but the configuration of a solution confined to a circle cannot expand or contract and remain similar to itself. Therefore the study of homographic solutions that are not relative equilibria is apparently more complicated than in the classical case, in which all homographic orbits
are planar.

We focus here on three types of homographic solutions. The first, which
we call Lagrangian, form an equilateral triangle at every time instant. We
ask that the plane of this triangle be always
orthogonal to the rotation axis. This assumption seems to be natural because, 
as proved in \cite{Diacu1}, Lagrangian relative equilibria, which are particular
homographic Lagrangian orbits, obey this property. We prove the existence of  homographic Lagrangian orbits in Section 3, and provide their complete classification in the case of equal masses in Section 4, for $\kappa>0$, and Section 5, for $\kappa<0$. Moreover, we show in Section 6 that Lagrangian solutions with non-equal masses don't exist. 

We then study another type of homographic solutions of the curved $3$-body problem, which we call Eulerian, in analogy with the classical case that refers to bodies confined to
a rotating straight line. At every time instant, the bodies of an Eulerian homographic orbit are on a (possibly) rotating geodesic. In Section 7 we prove the existence of these orbits. Moreover, for equal masses, we provide their complete classification in Section 8, for $\kappa>0$, and Section 9, for $\kappa<0$. 

Finally, in Section 10, we discuss the existence of hyperbolic homographic solutions,
which occur only for negative curvature. We prove that when the bodies are on the same hyperbolically rotating geodesic, a class of solutions we call hyperbolic Eulerian, every orbit is a hyperbolic Eulerian relative equilibrium. Therefore hyperbolic Eulerian relative equilibria are unstable, a fact that makes them unlikely observable candidates in a (hypothetically) hyperbolic physical universe.


\section{Equations of motion}\label{equations}

We consider the equations of motion on $2$-dimensional manifolds of
constant curvature, namely spheres embedded in $\mathbb{R}^3$, for $\kappa>0$, 
and hyperboloids\footnote{The hyperboloid corresponds to Weierstrass's model
of hyperbolic geometry (see Appendix in \cite{Diacu1}).} embedded in the Minkovski 
space ${\mathbb{M}}^3$, for $\kappa<0$. 

Consider the masses $m_1, m_2, m_3>0$ in $\mathbb{R}^3$, for $\kappa>0$,
and in $\mathbb{M}^3$, for $\kappa<0$, whose positions are given by the vectors
${\bf q}_i=(x_i,y_i,z_i), \ i=1, 2, 3$. Let ${\bf q}=
({\bf q}_1, {\bf q}_2,{\bf q}_3)$ be the configuration of the system,
and ${\bf p}=({\bf p}_1, {\bf p}_2,{\bf p}_3)$, with ${\bf p}_i=m_i\dot{\bf q}_i$,
representing the momentum.
We define the gradient operator with respect to the vector ${\bf q}_i$ as
$$\widetilde\nabla_{{\bf q}_i}=(\partial_{x_i},\partial_{y_i},\sigma\partial_{z_i}),$$
where $\sigma$ is the {\it signature function},
\begin{equation}
\sigma=
\begin{cases}
+1, \ \ {\rm for} \ \ \kappa>0\cr
-1, \ \ {\rm for} \ \ \kappa<0,\cr
\end{cases}\label{sigma}
\end{equation}
and let $\widetilde\nabla$ denote the operator
$(\widetilde\nabla_{{\bf q}_1},\widetilde\nabla_{{\bf q}_2},\widetilde\nabla_{{\bf q}_3})$.
For the 3-dimensional vectors ${\bf a}=(a_x,a_y,a_z)$ and ${\bf b}=(b_x,b_y,b_z)$, 
we define the inner product
\begin{equation}
{\bf a}\odot{\bf b}:=(a_xb_x+a_yb_y+\sigma a_zb_z)
\label{dotpr}
\end{equation}
and the cross product
\begin{equation}
{\bf a}\otimes{\bf b}:=(a_yb_z-a_zb_y, a_zb_x-a_xb_z, 
\sigma(a_xb_y-a_yb_x)).
\end{equation}

The Hamiltonian function of the system describing the motion of the $3$-body problem in spaces of constant curvature is
$$H_\kappa({\bf q},{\bf p})=T_\kappa({\bf q},{\bf p})-U_\kappa({\bf q}),$$
where 
$$
T_\kappa({\bf q},{\bf p})={1\over 2}\sum_{i=1}^3m_i^{-1}({\bf p}_i\odot{\bf p}_i)(\kappa{\bf q}_i\odot{\bf q}_i)
$$
defines the kinetic energy and
\begin{equation}
U_\kappa({\bf q})=\sum_{1\le i<j\le 3}{m_im_j
|\kappa|^{1/2}{\kappa{\bf q}_i\odot{\bf q}_j}\over
[\sigma(\kappa{\bf q}_i
\odot{\bf q}_i)(\kappa{\bf q}_j\odot{\bf q}_j)-\sigma({\kappa{\bf q}_i\odot{\bf q}_j
})^2]^{1/2}}
\label{forcef}
\end{equation}
is the force function, $-U_\kappa$ representing the potential energy\footnote{In \cite{Diacu1}, we showed how this expression of $U_\kappa$ follows from the cotangent potential for $\kappa\ne 0$, and that $U_0$ is the Newtonian potential of the Euclidean problem, obtained as $\kappa\to 0$.}. Then the Hamiltonian form of the equations of motion is given by the system
\begin{equation}
\begin{cases}
\dot{\bf q}_i=
m_i^{-1}{\bf p}_i,\cr
\dot{\bf p}_i=\widetilde\nabla_{{\bf q}_i}U_\kappa({\bf q})-m_i^{-1}\kappa({\bf p}_i\odot{\bf p}_i)
{\bf q}_i, \ \  i=1,2,3, \ \kappa\ne 0,
\label{Ham}
\end{cases}
\end{equation}
where the gradient of the force function has the expression
\begin{equation}
{\widetilde\nabla}_{{\bf q}_i}U_\kappa({\bf q})=\sum_{\substack{j=1\\ j\ne i}}^3{m_im_j|\kappa|^{3/2}(\kappa{\bf q}_j\odot{\bf q}_j)[(\kappa{\bf q}_i\odot{\bf q}_i){\bf q}_j-(\kappa{\bf q}_i\odot{\bf q}_j){\bf q}_i]\over
[\sigma(\kappa{\bf q}_i
\odot{\bf q}_i)(\kappa{\bf q}_j\odot{\bf q}_j)-\sigma({\kappa{\bf q}_i\odot{\bf q}_j
})^2]^{3/2}}.
\label{gradient}
\end{equation}
The motion is confined to the surface of nonzero constant curvature $\kappa$, i.e.\ $({\bf q},{\bf p})\in {\bf T}^*({\bf M}_\kappa^2)^3$,  where ${\bf T}^*({\bf M}_\kappa^2)^3$ is the cotangent bundle of the configuration space $({\bf M}^2_\kappa)^3$, and
$$
{\bf M}^2_\kappa=\{(x,y,z)\in\mathbb{R}^3\ |\ \kappa(x^2+y^2+\sigma z^2)=1\}.
$$ 
In particular, ${\bf M}^2_1={\bf S}^2$ is the 2-dimensional sphere, and ${\bf M}^2_{-1}={\bf H}^2$ is the 2-dimensional hyperbolic plane, represented by the upper sheet of the hyperboloid of two sheets (see the Appendix of \cite{Diacu1} for more details). 
We will also denote ${\bf M}^2_\kappa$ by ${\bf S}^2_\kappa$ for $\kappa>0$,
and by ${\bf H}^2_\kappa$ for $\kappa<0$.

Notice that the $3$ constraints given by $\kappa{\bf q}_i\odot{\bf q}_i=1, i=1,2,3,$ imply that ${\bf q}_i\odot{\bf p}_i=0$, so the $18$-dimensional system \eqref{Ham} has  $6$ constraints.
The Hamiltonian function provides the integral of energy,
$$
H_\kappa({\bf q},{\bf p})=h,
$$
where $h$ is the energy constant. Equations \eqref{Ham} also have the integrals 
of the angular momentum,
\begin{equation}
\sum_{i=1}^3{\bf q}_i\otimes{\bf p}_i={\bf c},\label{ang}
\end{equation}
where ${\bf c}=(\alpha, \beta, \gamma)$ is a constant vector. Unlike in the
Euclidean case, there are no integrals of the center of mass and linear
momentum. Their absence complicates the study of the problem
since many of the standard methods don't apply anymore.

Using the fact that $\kappa{\bf q}_i\odot{\bf q}_i=1$ for $i=1,2,3$, we can write system
\eqref{Ham} as
\begin{equation}
\ddot{\bf q}_i=\sum_{\substack{j=1\\ j\ne i}}^3{m_j|\kappa|^{3/2}[{\bf q}_j-(\kappa{\bf q}_i\odot{\bf q}_j){\bf q}_i]\over
[\sigma-\sigma({\kappa{\bf q}_i\odot{\bf q}_j
})^2]^{3/2}}-(\kappa\dot{\bf q}_i\odot\dot{\bf q}_i){\bf q}_i, \ \ i=1,2,3,
\label{second}
\end{equation}
which is the form of the equations of motion we will use in this paper.


\section{Local existence and uniqueness of Lagrangian solutions}

In this section we define the Lagrangian solutions of the curved 3-body
problem, which form a particular class of homographic orbits. Then, 
for equal masses and suitable initial conditions, we prove their local existence and uniqueness.

\begin{definition}
A solution of equations \eqref{second} is called Lagrangian if, at every time
$t$, the masses form an equilateral triangle that is orthogonal to the $z$ axis. 
\label{deflag}
\end{definition}

According to Definition \ref{deflag}, the size of a Lagrangian solution can vary,
but its shape is always the same. Moreover, all masses have the same coordinate 
$z(t)$, which may also vary in time, though the triangle is always
perpendicular to the $z$ axis. 

We can represent a Lagrangian solution of the curved 3-body problem 
in the form
\begin{equation}
{\bf q}
=({\bf q}_1,{\bf q}_2, {\bf q}_3), \ \ {\rm with}\ \ {\bf q}_i=(x_i,y_i,z_i),\  i=1,2,3,
\label{lagsol}
\end{equation}
\begin{align*}
x_1&=r\cos\omega,& y_1&=r\sin\omega,& z_1&=z,\\
x_2&=r\cos(\omega +2\pi/3),& y_2&=r\sin(\omega +2\pi/3),& z_2&=z,\\
x_3&=r\cos(\omega +4\pi/3),& y_3&=r\sin(\omega +4\pi/3),& z_3&=z,
\end{align*}
where $z=z(t)$ satisfies $z^2=\sigma\kappa^{-1}-\sigma r^2$; $\sigma$ is the signature
function defined in \eqref{sigma}; $r:=r(t)$ is the {\it size function}; and $\omega:=\omega(t)$ is the {\it angular function}.

Indeed, for every time $t$, we have that $x_i^2(t)+y_i^2(t)+\sigma z_i^2(t)=\kappa^{-1},\ i=1,2,3$, which means that the bodies stay on the surface ${\bf M}_{\kappa}^2$, each body has the same $z$ coordinate, i.e.\ the plane of the triangle is orthogonal to the $z$ axis, and the angles between any two bodies, seen from the geometric center of the triangle, are always the same, so the triangle remains equilateral. Therefore representation \eqref{lagsol} of the Lagrangian orbits agrees with Definition \ref{deflag}.

\begin{definition}
A Lagrangian solution of equations \eqref{second} is called Lagrangian homothetic if
the equilateral triangle expands or contracts, but does not rotate around the $z$ axis. 
\end{definition}

In terms of representation \eqref{lagsol}, a Lagrangian solution is Lagrangian
homothetic if $\omega(t)$ is constant, but $r(t)$ is not constant. Such orbits occur, for
instance, when three bodies of equal masses lying initially in the same open hemisphere are released with zero velocities from an equilateral configuration, to end up in a triple collision.

\begin{definition}
A Lagrangian solution of equations \eqref{second} is called a Lagrangian relative 
equilibrium if the triangle rotates around the $z$ axis without expanding or contracting. 
\end{definition}

In terms of representation \eqref{lagsol}, a Lagrangian relative equilibrium occurs when $r(t)$ is constant, but $\omega(t)$ is not constant. Of course, Lagrangian homothetic solutions and Lagrangian relative equilibria, whose existence we proved in \cite{Diacu1}, are particular Lagrangian orbits, but we expect that the Lagrangian orbits are not reduced to them. We  now show this by proving the local existence and uniqueness of Lagrangian solutions that are neither Lagrangian homothetic, nor Lagrangian relative equilibria.


\begin{theorem}
In the curved $3$-body problem of equal masses, for every set of initial conditions 
belonging to a certain class, the local existence and uniqueness of a Lagrangian solution,
which is neither Lagrangian homothetic nor a Lagrangian relative equilibrium, is assured. \label{equal-masses}
\end{theorem}
\begin{proof}
We will check to see if equations \eqref{second} admit solutions of the form \eqref{lagsol}
that start in the region $z>0$ and for which both $r(t)$ and $\omega(t)$ are not constant. We compute then that
$$\kappa{\bf q}_i\odot{\bf q}_j=1-3\kappa r^2/2\ \ {\rm for}\ \ i,j=1,2,3, \ \ {\rm with} \ \ i\ne j,$$
\begin{align*}
\dot x_1&=\dot r\cos\omega-r\dot\omega\sin\omega,& \dot y_1&=\dot r\sin\omega+
r\dot\omega\cos\omega,
\end{align*}
$$\dot x_2=\dot r\cos\Big(\omega +{2\pi\over 3}\Big)-r\dot\omega\sin\Big(\omega +{2\pi\over 3}\Big),$$ 
$$\dot y_2=\dot r\sin\Big(\omega +{2\pi\over 3}\Big)+r\dot\omega\cos\Big(\omega +{2\pi\over 3}\Big),$$
$$\dot x_3=\dot r\cos\Big(\omega +{4\pi\over 3}\Big)-r\dot\omega\sin\Big(\omega +{4\pi\over 3}\Big),$$ 
$$\dot y_3=\dot r\sin\Big(\omega +{4\pi\over 3}\Big)+r\dot\omega\cos\Big(\omega +{4\pi\over 3}\Big),$$
\begin{equation}
\dot z_1=\dot z_2=\dot z_3=-\sigma r\dot r(\sigma\kappa^{-1}-\sigma r^2)^{-1/2},
\label{zeds}
\end{equation}
$$\kappa\dot{\bf q}_i\odot\dot{\bf q}_i=\kappa r^2\dot\omega^2+{\kappa\dot r^2\over
1-\kappa r^2} \ \ {\rm for}\ \ i=1,2,3,$$
$$\ddot x_1=(\ddot r-r\dot\omega^2)\cos\omega-(r\ddot\omega+2\dot r\dot\omega)\sin\omega,$$
$$\ddot y_1=(\ddot r-r\dot\omega^2)\sin\omega+(r\ddot\omega+2\dot r\dot\omega)\cos\omega,$$
$$\ddot x_2=(\ddot r-r\dot\omega^2)\cos\Big(\omega +{2\pi\over 3}\Big)-(r\ddot\omega+2\dot r\dot\omega)\sin\Big(\omega +{2\pi\over 3}\Big),$$
$$\ddot y_2=(\ddot r-r\dot\omega^2)\sin\Big(\omega +{2\pi\over 3}\Big)+(r\ddot\omega+2\dot r\dot\omega)\cos\Big(\omega +{2\pi\over 3}\Big),$$
$$\ddot x_3=(\ddot r-r\dot\omega^2)\cos\Big(\omega +{4\pi\over 3}\Big)-(r\ddot\omega+2\dot r\dot\omega)\sin\Big(\omega +{4\pi\over 3}\Big),$$
$$\ddot y_3=(\ddot r-r\dot\omega^2)\sin\Big(\omega +{4\pi\over 3}\Big)+(r\ddot\omega+2\dot r\dot\omega)\cos\Big(\omega +{4\pi\over 3}\Big),$$
$$\ddot z_1=\ddot z_2=\ddot z_3=-\sigma r\ddot r(\sigma\kappa^{-1}-\sigma r^2)^{-1/2}-
\kappa^{-1}\dot r^2(\sigma\kappa^{-1}-\sigma r^2)^{-3/2}.$$
Substituting these expressions into system \eqref{second}, we are led to the
system below, where the double-dot terms on the left indicate to which differential
equation each algebraic equation corresponds:
\begin{align*}
\ddot x_1: \ \ \ \ \ \ \ \ & A\cos\omega-B\sin\omega=0,\\
\ddot x_2:\ \ \ \ \ \ \ \ & A\cos\Big(\omega +{2\pi\over 3}\Big)-B\sin\Big(\omega +{2\pi\over 3}\Big)=0,\\
\ddot x_3:\ \ \ \ \ \ \ \ & A\cos\Big(\omega +{4\pi\over 3}\Big)-B\sin\Big(\omega +{4\pi\over 3}\Big)=0,\\
\ddot y_1:\ \ \ \ \ \ \ \ & A\sin\omega+B\cos\omega=0,\\
\ddot y_2:\ \ \ \ \ \ \ \ & A\sin\Big(\omega +{2\pi\over 3}\Big)+B\cos\Big(\omega +{2\pi\over 3}\Big)=0,\\
\ddot y_3:\ \ \ \ \ \ \ \ & A\sin\Big(\omega +{4\pi\over 3}\Big)+B\cos\Big(\omega +{4\pi\over 3}\Big)=0,\\
\ddot z_1, \ddot z_2, \ddot z_3:\ \ \ \ \ \ \ \ & A=0,
\end{align*}
where
$$A:=A(t)=\ddot r-r(1-\kappa r^2)\dot\omega^2+{\kappa r\dot r^2\over{1-\kappa r^2}}+
{24m(1-\kappa r^2)\over{r^2(12-9\kappa r^2)^{3/2}}},$$
$$B:=B(t)=r\ddot\omega+2\dot r\dot\omega.$$
Obviously, the above system has solutions if and only if $A=B=0$,
which means that the local existence and uniqueness of Lagrangian orbits with equal masses is equivalent to the existence of solutions of the system of differential equations
\begin{equation}
\begin{cases}
\dot r=\nu\cr
\dot w=-{2\nu w\over r}\cr
\dot\nu=r(1-\kappa r^2)w^2-{\kappa r\nu^2\over{1-\kappa r^2}}-
{24m(1-\kappa r^2)\over{r^2(12-9\kappa r^2)^{3/2}}}, \cr
\end{cases}
\label{prime}
\end{equation}
with initial conditions $r(0)=r_0, w(0)=w_0, \nu(0)=\nu_0,$ where $w=\dot\omega$.
The functions $r,\omega$, and $w$ are analytic, and as long as the initial
conditions satisfy the conditions $r_0>0$ for all $\kappa$, as well as $r_0<\kappa^{-1/2}$
for $\kappa>0$, standard results of the theory of differential equations guarantee the
local existence and uniqueness of a solution $(r,w,\nu)$ of equations \eqref{prime},
and therefore the local existence and uniqueness of a Lagrangian orbit with
$r(t)$ and $\omega(t)$ not constant. The proof is now complete.
\end{proof}


\section{Classification of Lagrangian solutions for $\kappa>0$}

We can now state and prove the following result:

\begin{theorem}
In the curved $3$-body problem with equal masses and $\kappa>0$ there are five classes of Lagrangian solutions:

(i) Lagrangian homothetic orbits that begin or end in total collision in finite time;

(ii) Lagrangian relative equilibria that move on a circle; 

(iii) Lagrangian periodic orbits that are neither Lagrangian homothetic nor Lagrangian relative equilibria;

(iv) Lagrangian non-periodic, non-collision orbits that eject at time $-\infty$,
with zero velocity, from the equator, reach a maximum distance from the equator, which depends on the initial conditions, and return to the equator, with zero velocity, at time
$+\infty$. 

None of the above orbits can cross the equator, defined as the great circle of the
sphere orthogonal to the $z$ axis.

(v) Lagrangian equilibrium points, when the three equal masses are fixed on the 
equator at the vertices of an equilateral triangle.

\label{homo}
\end{theorem}

The rest of this section is dedicated to the proof of this theorem.

Let us start by noticing that the first two equations of system \eqref{prime} 
imply that $\dot w=-{2\dot r w\over r}$, which leads to 
$$w=\frac{c}{r^2},$$
where $c$ is a constant. The case $c=0$ can occur only when $w=0$,
which means $\dot\omega=0$. Under these circumstances the angular
velocity is zero, so the motion is homothetic. These are the orbits whose
existence is stated in Theorem \ref{homo} (i). They occur 
only when the angular momentum is zero, and lead to a triple collision
in the future or in the past, depending on the sense of the velocity vectors.

For the rest of this section, we assume that $c\ne 0$. Then system \eqref{prime}
takes the form
\begin{equation}\label{lag2}
\begin{cases}
\dot r=\nu\cr 
\dot \nu=\frac{c^2(1-\kappa r^2)}{r^3} -{\kappa
r\nu^2\over{1-\kappa r^2}}- {24m(1-\kappa r^2)\over{r^2(12-9\kappa
r^2)^{3/2}}}. \cr
\end{cases}
\end{equation}

Notice that the term ${\kappa r\nu^2\over{1-\kappa r^2}}$ of the last equation arises from the derivatives $\dot z_1, \dot z_2, \dot z_3$ in \eqref{zeds}. But these derivatives would be zero if the equilateral triangle rotates along the equator, because $r$ is constant in this case, so the term ${\kappa r\nu^2\over{1-\kappa r^2}}$ vanishes. Therefore the existence of equilateral relative equilibria on the equator (included in statement (ii) above), and the
existence of equilibrium points (stated in (v))---results proved in \cite{Diacu1}---are in agreement with the above equations. Nevertheless, the term ${\kappa r\nu^2\over{1-\kappa r^2}}$ stops any orbit from crossing the equator, a fact mentioned before statement (v) of Theorem \ref{homo}.

Understanding system \eqref{lag2} is the key to proving Theorem \ref{homo}. We start
with the following facts:


\begin{lemma}
Assume $\kappa, m>0$ and $c\ne 0$. Then for $\kappa^{1/2}c^2-(8/\sqrt{3})m< 0$, system \eqref{lag2} has two fixed points, while for $\kappa^{1/2}c^2-(8/\sqrt{3})m\ge 0$ it has one fixed point.
\label{prima}
\end{lemma}

\begin{proof}
The fixed points of system \eqref{lag2} are given by $\dot r=0=\dot \nu.$ 
Substituting
$\nu=0$ in the second equation of (\ref{lag2}), we obtain
$$\frac{1-\kappa r^2}{r^2}\left[\frac{c^2}{r} - \frac{24m}{(12-9\kappa
r^2)^{3/2}}\right] = 0.$$
The above remarks show that, for $\kappa>0$, $r=\kappa^{-1/2}$ is a fixed point, which
physically represents an equilateral relative equilibrium moving along the
equator. Other potential fixed points of system \eqref{lag2} are
given by the equation
$$c^2(12-9\kappa r^2)^{3/2} = 24mr,$$ whose solutions
are the roots of the polynomial
\begin{equation}\label{polynomial1}
729c^4\kappa^3r^6 - 2916c^4\kappa^2r^4 + 144(27c^4\kappa +
4m^2)r^2 - 1728.
\end{equation}
Writing $x=r^2$ and assuming  $\kappa >0$, this polynomial
takes the form
\begin{equation}\label{polynomial2}
p(x)=729\kappa^3x^3 - 2916c^4\kappa^2x^2 +
144(27c^4\kappa + 4m^2)x - 1728,
\end{equation}
and its derivative is given by
\begin{equation}\label{dpolynomial2}
p'(x)=2187c^4\kappa^3x^2 - 5832c^4\kappa^2x +
144(27c^4\kappa + 4m^2).
\end{equation}
The discriminant of $p'$ is
$-5038848c^4\kappa^3m^2<0.$

By Descartes's rule of signs, $p$ can have one or three positive roots. If $p$ 
has three positive roots, then $p'$ must have two positive roots, but 
this is not possible because its discriminant is negative. Consequently $p$ 
has exactly one positive root.

For the point $(r,\nu)=(r_0,0)$ to be a fixed point of
equations \eqref{lag2}, $r_0$ must satify the
inequalities $0<r_0\le\kappa^{-1/2}$. If we denote
\begin{equation}
g(r)=\frac{c^2}{r} - \frac{24m}{(12-9\kappa
r^2)^{3/2}},\label{g}
\end{equation}
we see that, for $\kappa>0$, $g$ is a decreasing function since
\begin{equation}
{d\over dr}g(r)=-{c^2\over r^2}-{648 m\kappa r\over (12-9\kappa r^2)^{5/2}}<0.
\label{derivg}
\end{equation}
When $r\to 0$, we obviously have that $g(r)>0$ since we assumed $c\ne 0$.
When $r\to \kappa^{-1/2}$, we have $g(r)\to \kappa^{1/2}c^2-(8/\sqrt{3})m$.
If $\kappa^{1/2}c^2-(8/\sqrt{3})m>0$, then $r_0>\kappa^{-1/2}$, so $(r_0,0)$
is not a fixed point. Therefore, assuming $c\ne 0$, a necessary condition that 
$(r_0,0)$ is a fixed point of system \eqref{lag2} with $0<r_0<\kappa^{-1/2}$  is that
$$\kappa^{1/2}c^2-(8/\sqrt{3})m< 0.$$
For $\kappa^{1/2}c^2-(8/\sqrt{3})m\ge 0,$ the only fixed point of
system \eqref{lag2} is $(r,\nu)=(\kappa^{-1/2},0)$. This conclusion completes the proof
of the lemma.
\end{proof}

\subsection{\bf The flow in the $(r,\nu)$ plane for $\kappa>0$} We will
now study the flow of system \eqref{lag2} in the $(r,\nu)$ plane for $\kappa>0$.
At every point with $\nu\ne 0$, the slope of the vector field is given by 
${d\nu\over dr}$, i.e.\ by the ratio ${\dot\nu\over\dot r}=h(r,\nu),$ where 
$$h(r,\nu)= \frac{c^2(1-\kappa r^2)}{\nu r^3} -{\kappa
r\nu\over{1-\kappa r^2}}- {24m(1-\kappa r^2)\over{\nu r^2(12-9\kappa
r^2)^{3/2}}}.$$
Since $h(r,-\nu)=-h(r,\nu)$, the flow of system \eqref{lag2} is symmetric with
respect to the $r$ axis for $r\in(0,\kappa^{-1/2}]$. Also notice that, except for 
the fixed point $(\kappa^{-1/2},0)$, system \eqref{lag2} is undefined on the 
lines $r=0$ and $r=\kappa^{-1/2}$. Therefore the flow of system
\eqref{lag2}  exists only for points $(r,\nu)$ in the band $(0,\kappa^{-1/2})\times \mathbb{R}$ and for the point $(\kappa^{-1/2},0)$.

Since $\dot r=\nu$, no interval on the $r$ axis can be an invariant set for
system \eqref{lag2}. Then the symmetry of the flow relative to the $r$ axis implies that 
orbits cross the $r$ axis perpendicularly. But since $g(r)\ne 0$ at every non-fixed point, 
the flow crosses the $r$ axis perpendicularly everywhere, except at the fixed points. 

Let us further treat the case of one fixed point and the case of two
fixed points separately.

\subsubsection{\bf The case of one fixed point} 
A single fixed point, namely $(\kappa^{-1/2},0)$, appears
when $\kappa^{1/2}c^2-(8/\sqrt{3})m\ge 0$. Then the function $g$, which is 
decreasing, has no zeroes for $r\in(0,\kappa^{-1/2})$, therefore 
$g(r)>0$ in this interval, so the flow always crosses the $r$ axis upwards. 

For $\nu\ne 0$, the right hand side of the second
equation of \eqref{lag2} can be written as
\begin{equation}
G(r,\nu)=g_1(r)g(r)+g_2(r,\nu),
\label{g*}
\end{equation}
where 
\begin{equation}
g_1(r)=\frac{1-\kappa r^2}{r^2}\ \ 
{\rm and}\ \ 
g_2(r,\nu)=-{\kappa r\nu^2\over 1-\kappa r^2}.
\label{g12}
\end{equation}
 But $\frac{d}{dr}g_1(r)=-2/r^3<0$
and $\frac{\partial}{\partial r}g_2(r,\nu)=-{\kappa\nu^2(1+\kappa r^2)\over(1-\kappa r^2)^2}<0.$
So, like $g$, the functions $g_1$ and $g_2$ are decreasing in $(0,\kappa^{-1/2})$, 
with $g_1, g>0$, therefore $G$ is a decreasing function as well.
Consequently, for $\nu= {\rm constant} >0$, the slope of the vector field
decreases from $+\infty$ at $r=0$ to $-\infty$ ar $r=\kappa^{-1/2}$. For
$\nu= {\rm constant}<0$, the slope of the vector field increases from
$-\infty$ at $r=0$ to $+\infty$ at $r=\kappa^{-1/2}$.

This behavior of the vector field forces every orbit to eject downwards 
from the fixed point, at time $t=-\infty$ and with zero velocity, on a trajectory tangent to the 
line $r=\kappa^{-1/2}$, reach slope zero at some moment in time, then cross the 
$r$ axis perpendicularly upwards and symmetrically return with final zero velocity, 
at time $t=+\infty$, to the fixed point (see Figure \ref{Fig1}(a)). So the flow of system  
\eqref{lag2} consists in this case solely of homoclinic orbits to the fixed point $(\kappa^{-1/2},0)$, orbits whose existence is claimed in Theorem \ref{homo} (iv).
Some of these trajectories may come very close to a total collapse, which
they will never reach because only solutions with zero angular momentum
(like the homothetic orbits) encounter total collisions, as proved in \cite{Diacu1bis}. 

So the orbits cannot reach any singularity of the line $r=0$, and neither can
they begin or end in a singularity of the line $r=\kappa^{-1/2}$. The reason 
for the latter is that such points are or the form $(\kappa^{-1/2},\nu)$ with
$\nu\ne 0$, therefore $\dot r\ne 0$ at such points. But the vector field tends
to infinity when approaching the line $r=\kappa^{-1/2}$, so the flow must 
be tangent to it, consequently $\dot r$ must tend to zero, which is a contradiction.
Therefore only homoclinic orbits exist in this case.

\begin{figure}[htbp] 
   \centering
   \includegraphics[width=2in]{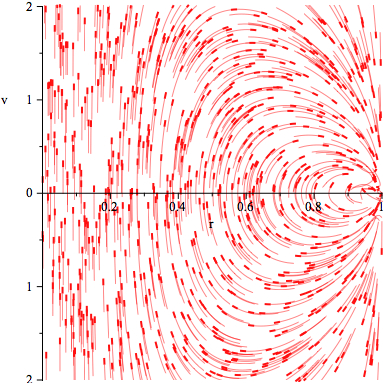}
    \includegraphics[width=2in]{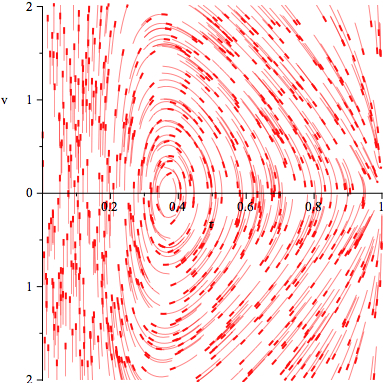}
   \caption{A sketch of the flow of system \eqref{lag2} for (a) $\kappa=c=1,m=0.24$, typical for one fixed point, and (b) $\kappa=c=1, m=4$, typical for two fixed points.}
   \label{Fig1}
\end{figure}


\subsubsection{\bf The case of two fixed points} Two fixed points, 
$(\kappa^{-1/2},0)$ and $(r_0,0)$, with $0<r_0<\kappa^{-1/2}$, occur when
$\kappa^{1/2}c^2-(8/\sqrt{3})m< 0$.
Since $g$ is decreasing in the interval $(0,\kappa^{-1/2})$, we
can conclude that 
$g(t)>0$ for $t\in(0,r_0)$ and $g(t)<0$ for $t\in(r_0,\kappa^{-1/2})$.
Therefore the flow of system \eqref{lag2} crosses the $r$ axis
upwards when $r<r_0$, but downwards for $r>r_0$ (see Figure \ref{Fig1}(b)). 

The function $G(r,\nu)$, defined in \eqref{g*}, fails to be decreasing in the interval $(0,\kappa^{-1/2})$ along lines of constant $\nu$, but it has no singularities in this interval and still maintains the properties
$$\lim_{r\to 0^+}G(r,\nu)=+\infty \ \ {\rm and}\ \lim_{r\to (\kappa^{-1/2})^-}G(r,\nu)=
-\infty.$$
Therefore $G$ must vanish at some point, so due to the symmetry of the vector
field with respect to the $r$ axis, the fixed point $(r_0,0)$ is surrounded by
periodic orbits. The points where $G$ vanishes are given by the nullcline $\dot\nu=0$,
which has the expression
$$\nu^2={(1-\kappa r^2)^2\over\kappa r^3}\bigg[\frac{c^2}{r}-
{24m\over{(12-9\kappa r^2)^{3/2}}}\bigg].$$
This nullcline is a disconnected set, formed by the fixed point $(\kappa^{-1/2},0)$ and a continuous curve, symmetric with
respect to the $r$ axis. Indeed, since
the equation of the nullcline can be written as $\nu^2={(1-\kappa r^2)^2\over\kappa r^3}
g(r)$, and $\lim_{r\to(\kappa^{-1/2})^-}g(r)=\kappa^{1/2}c^2-(8/\sqrt{3})m<0$ in the case of two fixed points (as shown in the proof of Lemma \ref{prima}), only the point $(\kappa^{-1/2},0)$ satisfies
the nullcline equation away from the fixed point $(r_0,0)$.

The asymptotic behavior of $G$ near $r=\kappa^{-1/2}$ also forces the flow to produce homoclinic orbits for the fixed point $(\kappa^{-1/2},0)$, as in the case discussed in Subsection 4.1.1. The existence of these two kinds of solutions is stated in Theorem \ref{homo} (iii) and (iv), respectively. The fact that orbits cannot begin or end at any
of the singularities of the lines $r=0$ or $r=\kappa^{-1/2}$ follows as in
Subsection 4.1.1. This remark completes the proof of Theorem 2.


\section{Classification of Lagrangian solutions for $\kappa<0$}

We can now state and prove the following result:

\begin{theorem}
In the curved $3$-body problem with equal masses and $\kappa<0$ there are eight 
classes of Lagrangian solutions:

(i) Lagrangian homothetic orbits that begin or end in total collision in finite time;

(ii) Lagrangian relative equilibria, for which the bodies move on a circle parallel with the $xy$ plane; 

(iii) Lagrangian periodic orbits that are not Lagrangian relative equilibria;

(iv) Lagrangian orbits that eject at time $-\infty$ from a certain relative equilibrium solution $\bf s$ (whose existence and position depend on the values of the parameters) and returns to it at time $+\infty$;

(v) Lagrangian orbits that come from infinity at time $-\infty$ and reach the
relative equilibrium $\bf s$ at time $+\infty$;

(vi) Lagrangian orbits that eject from the relative equilibrium $\bf s$
at time $-\infty$ and reach infinity at time $+\infty$;

(vii) Lagrangian orbits that come from infinity at time $-\infty$ and symmetrically return to infinity at time $+\infty$, never able to reach the
Lagrangian relative equilibrium $\bf s$;

(viii)  Lagrangian orbits that come from infinity at time $-\infty$, reach a position close to a total collision, and symmetrically return to infinity at time $+\infty$.

\label{homoneg}
\end{theorem}

The rest of this section is dedicated to the proof of this theorem. Notice
first that the orbits described in Theorem \ref{homoneg} (i) occur for
zero angular momentum, when $c=0$, as for instance when the
three equal masses are released with zero velocities
from the Lagrangian configuration, a case in which a total collapse
takes place at the point $(0,0,|\kappa|^{-1/2})$. Depending on the initial
conditions, the motion can be bounded or unbounded. 
The existence of the orbits described in
Theorem \ref{homoneg} (ii) was proved in \cite{Diacu1}.
To address the other points of Theorem \ref{homoneg}, and show that no other
orbits than the ones stated there exist, we need to study the flow of system \eqref{lag2}
for $\kappa<0$. Let us first prove the following fact:


\begin{lemma}
Assume $\kappa<0, m>0$, and $c\ne 0$. Then system \eqref{lag2} has no fixed points
when $27c^4\kappa+4m^2\le 0$, and can have two, one, or no fixed points when $27c^4\kappa+4m^2> 0$.
\end{lemma}

\begin{proof}
The number of fixed points of system \eqref{lag2} is the same as
the number of positive zeroes of the polynomial $p$ defined in \eqref{polynomial2}.
If $27c^4\kappa+4m^2\le 0$, all coefficients of $p$ are negative, so
by Descartes's rule of signs, $p$ has no positive roots.

Now assume that $27c^4\kappa+4m^2> 0$. Then the zeroes of $p$ are 
the same as the zeroes of the monic polynomial (i.e.\ with leading coefficient 1):
$$
{\bar p}(x)=x^3-4\kappa^{-1}x^2+[48\kappa^{-2}+(64/81)c^{-4}\kappa^{-3}m^2]x-(64/27)\kappa^{-3},$$
obtained when dividing $p$ by the leading coefficient.
But a monic cubic polynomial can be written as
$$x^3 - (a_1+a_2+a_3)x^2 + (a_1a_2+a_2a_3+a_3a_1)x - a_1a_2a_3,$$
where $a_1,a_2,$ and $a_3$ are its roots. One of
these roots is always real and has the opposite sign of $-a_1a_2a_3$.
Since the free term of $\bar p$ is positive, one of
its roots is always negative, independently of the allowed values of the coefficients
$\kappa, m, c$. Consequently $p$ can have two positive roots (including the
possibility of a double positive root) or no positive root at all. Therefore system 
\eqref{lag2} can have two, one, or no fixed points. As we will see later, all three 
cases occur.
\end{proof}

We further state and prove a property, which we will use to establish Lemma \ref{bigG}:

\begin{lemma}
Assume $\kappa<0, m>0, c\ne 0$, let $(r_*,0)$ be a fixed point of
system \eqref{lag2}, and consider the function $g$ defined in \eqref{g}. Then ${d\over dr}g(r_*)=0$ if and only if $r_*=(-{2\over 3\kappa})^{1/2}$. Moreover, ${d^2\over dr^2}g(r_*)>0$.
\label{help}
\end{lemma}
\begin{proof}
Since $(r_*,0)$ is a fixed point of system \eqref{lag2}, it follows that
$g(r_*)=0$. Then it follows from relation \eqref{g} that $(12-9\kappa r_*^2)^{3/2}=24mr_*/c^2$. Substituting this value of $(12-9\kappa r_*^2)^{3/2}$ into the equation ${d\over dr}g(r_*)=0$, which is equivalent to
$${648m\kappa r_*\over (12-9\kappa r_*^2)^{5/2}}=-{c^2\over r_*^2},$$
 it follows that $27\kappa/(12-9\kappa r_*^2)=-1/r_*^2$. Therefore $r_*=(-{2\over 3\kappa})^{1/2}$. Obviously, for this value of $r_*$, $g(r_*)=0$, so the first part
 of Lemma \ref{help} is proved. To prove the second part, substitute 
 $r_*=(-{2\over 3\kappa})^{1/2}$ into the equation $g(r_*)=0$, which is
 then equivalent with the relation
 \begin{equation}
 9\sqrt{3}c^2(-\kappa)^{1/2}-4m=0.
 \label{intermediate}
 \end{equation}
 Notice that 
 $${d^2\over dr^2}g(r)={2c^2\over r^3}-{648m\kappa\over (12-9\kappa r^2)^{5/2}}-
{29160m\kappa^2r^2\over(12-9\kappa r^2)^{7/2}}.$$
Substituting for  $r_*=(-{2\over 3\kappa})^{1/2}$ in the above equation, and using \eqref{intermediate}, we are led to the conclusion that ${d^2\over dr^2}g(r_*)=-(2\sqrt{3}+6\sqrt{2})m\kappa/9\sqrt{6}$, which is positive for $\kappa<0$. This 
completes the proof.
\end{proof}

The following result is important for understanding a qualitative aspect of the flow
of system \eqref{lag2}, which we will discuss later in this section.

\begin{lemma}
Assume $\kappa<0, m>0, c\ne 0$, and let $(r_*,0)$ be a fixed point of
system \eqref{lag2}. If ${\partial\over\partial r}G(r_*,0)=0$, then ${\partial^2\over\partial r^2}G(r_*,0)>0$, where
$G$ is defined in \eqref{g*}.
\label{bigG}
\end{lemma}
\begin{proof}
Since $(r_*,0)$ is a fixed point of \eqref{lag2}, $G(r_*,0)=0$. But for $\kappa<0$,
we have $g_1(r_*)>0$, so necessarily $g(r_*)=0$. Moreover,
${d\over dr}g_1(r_*)\ne 0,$ and since ${\partial\over \partial r}g_2(r,\nu)=
-{{\kappa\nu^2(1+\kappa r^2)}\over {(1-\kappa r^2)^2}}$, it follows that ${\partial\over \partial r}g_2(r_*,0)=0$.
But 
$${\partial G\over \partial r}(r,\nu)={d\over dr}g_1(r)\cdot g(r)+g_1(r){d\over dr}g(r)+{\partial\over \partial r}
g_2(r,\nu),$$
so the condition ${\partial\over \partial r}G(r_*,0)=0$ implies that ${d\over dr}g(r_*)=0$.
By Lemma \ref{help}, $r_*=(-{2\over 3\kappa})^{1/2}$ and ${d^2\over dr^2}g(r_*)>0$.
Using now the fact that
$${\partial^2 G\over \partial r^2}(r,\nu)={d^2\over dr^2}g_1(r)g(r)+2{d\over dr}g_1(r){d\over dr}g(r)+
g_1(r){d^2\over dr^2}g(r)+{\partial^2\over \partial r^2}g_2(r,\nu),$$
it follows that ${\partial^2\over \partial r^2}G(r_*,0)=g_1(r_*){d^2\over dr^2}g(r_*).$
Since Lemma \ref{help} implies that ${d^2\over dr^2}g(r_*)>0$, and we know that $g_1(r_*)>0$, it follows that ${\partial^2\over\partial r^2}G(r_*,0)>0$, a conclusion
that completes the proof. 
\end{proof}

\subsection{The flow in the $(r,\nu)$ plane for $\kappa<0$} We will now study the
flow of system \eqref{lag2} in the $(r,\nu)$ plane for $\kappa<0$. As in the case
$\kappa>0$, and for the same reason, the flow is symmetric with respect to the $r$ axis,
which it crosses perpendicularly at every non-fixed point with $r>0$. Since we can
have two, one, or no fixed points, we will treat each case separately.

\subsubsection{\bf The case of no fixed points} No fixed points occur when $g(r)$
has no zeroes. Since $g(r)\to\infty$ as $r\to 0$ with $r>0$, it follows that $g(r)>0$.
Since $g_1(r)$ and $g_2(r,\nu)$ are also positive, it follows that $G(r,\nu)>0$
for $r>0$. But $h(r,\nu)=G(r,\nu)/\nu$. Then $h(r,\nu)>0$ for $\nu>0$ and  $h(r,\nu)<0$ for $\nu<0$, so the flow
comes from infinity at time $-\infty$, crosses the $r$ axis perpendicularly upwards,
and symmetrically reaches infinity at time $+\infty$ (see Figure \ref{Fig2}(a)). These
are orbits as in the statements of Theorem \ref{homoneg} (vii) and (viii) but without any reference to the Lagrangian relative equilibrium $\bf s$.

\begin{figure}[htbp] 
   \centering
   \includegraphics[width=2in]{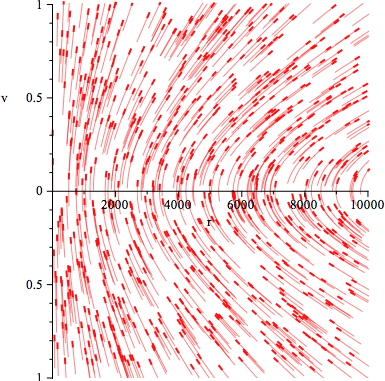}
   \includegraphics[width=2in]{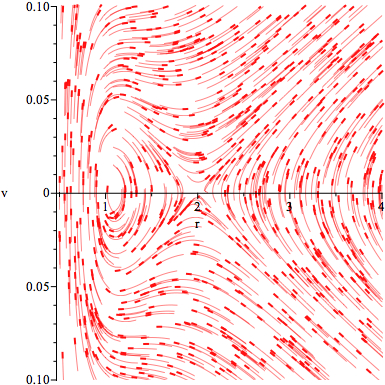}
   \caption{A sketch of the flow of system \eqref{lag2} for (a) 
   $\kappa = -2, c = 1/3,$ and  $m = 1/2$, typical for no fixed points;
   (b) $\kappa = -0.3, c = 0.23,$ and  $m = 0.12$, typical for two fixed points, which are in this case on the line $\nu=0$ at approximately $r_1=1.0882233$ and $r_2=2.0007055$.}
   \label{Fig2}
\end{figure}


\subsubsection{\bf The case of two fixed points} In this case, the function $g$ defined
at \eqref{g} has two distinct zeroes, one for $r=r_1$ and the other for $r=r_2$, with
$0<r_1<r_2$. In Theorem \ref{homoneg}, we denoted the fixed point $(r_2,0)$
by $\bf s$. Moreover, $g(r)>0$ for $r\in(0,r_1)\cup(r_2,\infty)$, and $g(r)<0$ for
$r\in(r_1,r_2)$. Therefore the vector field crosses the $r$ axis downwards between  $r_1$ and $r_2$, but upwards for $r<r_1$ as well as for $r>r_2$.

To determine the behavior of the flow near the fixed point $(r_1,0)$, we linearize system \eqref{lag2}. For this let $F(r,\nu)=\nu$ to be the right hand side of the first equation
in \eqref{lag2}, and notice that
${\partial F\over\partial r}(r_1,0)=1$, ${\partial F\over\partial\nu}(r_1,0)=0$,
and ${\partial G\over\partial\nu}(r_1,0)=0$. Since, along the $r$ axis, 
$G(r,0)$ is positive for $r<r_1$, but negative for $r_1<r<r_2$, 
it follows that either ${\partial G\over\partial r}(r_1,0)<0$ or ${\partial G\over\partial r}(r_1,0)=0$. But according to Lemma \ref{bigG}, if ${\partial G\over\partial r}(r_1,0)=0$, 
then ${\partial^2\over\partial r^2}G(r_*,0)>0$, so $G(r,0)$ is convex up at $(r_1,0)$. Then
$G(r,0)$ cannot not change sign when $r$ passes through $r_1$ along the
line $\nu=0$, so the only existing possibility is ${\partial G\over\partial r}(r_1,0)<0$.

The eigenvalues
of the linearized system corresponding to the fixed point $(r_1,0)$ are then
given by the equation
\begin{equation}
\det\begin{bmatrix}
-\lambda & 1\\
{\partial G\over\partial r}(r_1,0) & -\lambda
\end{bmatrix}
=0.\label{eigenv}
\end{equation}
Since ${\partial G\over\partial r}(r_1,0)$ is negative, the eigenvalues are purely imaginary, so $(r_1,0)$ is not a hyperbolic fixed point for equations \eqref{lag2}. Therefore this fixed point could be a spiral sink, a spiral source, or a center for the nonlinear system. But the symmetry of the flow of system \eqref{lag2} with respect to the $r$ axis, and the fact that, near $r_1$, the flow crosses the $r$ axis upwards to the left of $r_1$, and downwards to the right of $r_1$, eliminates the possibility of spiral behavior, so $(r_1,0)$ is a center (see Figure \ref{Fig2}(b)).

We can understand the generic behavior of the flow near the isolated
fixed point $(r_2,0)$ through linearization as well. For this purpose, notice that
${\partial F\over\partial r}(r_2,0)=1$, ${\partial F\over\partial\nu}(r_2,0)=0$,
and ${\partial G\over\partial\nu}(r_2,0)=0$. Since, along the $r$ axis, 
$G(r)$ is negative for $r_1<r<r_2$, but positive for $r>r_2$, 
it follows that  ${\partial G\over\partial r}(r_2,0)>0$ or
${\partial G\over\partial r}(r_2,0)=0$. But using Lemma \ref{bigG} the
same way we did above for the fixed point $(r_1,0)$, we can conclude
that the only possibility is ${\partial G\over\partial r}(r_2,0)>0$.

The eigenvalues corresponding to 
the fixed point $(r_2,0)$ are given by the equation
\begin{equation}
\det\begin{bmatrix}
-\lambda & 1\\
{\partial G\over\partial r}(r_2,0) & -\lambda
\end{bmatrix}
=0.\label{eigenv2}
\end{equation}
Consequently the fixed point $(r_2,0)$ is hyperbolic, its two eigenvalues are $\lambda_1>0$ and $\lambda_2<0$, so $(r_2,0)$ is a saddle. 

Indeed, for small $\nu>0$, the slope of the vector field decreases to $-\infty$ on 
lines $r=$ constant, with $r_1<r<r_2$, when $\nu$ tends to $0$. On the same 
lines, with $r>r_2$, the slope decreases from $+\infty$ as $\nu$ increases. This
behavior gives us an approximate idea of how the eigenvectors corresponding
to the eigenvalues $\lambda_1$ and $\lambda_2$ are positioned in the $r\nu$ plane.

On lines of the form $\nu=\eta r$, with $\eta>0$, the slope $h(r,\nu)$ of the
vector field becomes
$$h(r,\eta r)= \frac{1-\kappa r^2}{\eta r^3}\bigg[{c^2\over r}-{24m(1-\kappa r^2)\over{(12-9\kappa
r^2)^{3/2}}}\bigg] -{\kappa\eta
r^2\over{1-\kappa r^2}}.$$
So, as $r$ tends to $\infty$, the slope $h(r,\eta r)$ tends to $\eta$. Consequently 
the vector field doesn't bound the flow with negative slopes, and thus 
allows it to go to infinity.

With the fixed point $(r_1,0)$ as a center, the fixed point $(r_2,0)$ as a saddle, 
and a vector field that doesn't bound the orbits as $r\to\infty$, the flow must 
behave qualitatively as in Figure \ref{Fig2}(b).

This behavior of the flow proves the existence of the following types of solutions:

(a) periodic orbits around the fixed point $(r_1,0)$, corresponding to Theorem \ref{homoneg} (iii);

(b) a homoclinic orbit to the fixed point $(r_1,0)$, corresponding to Theorem \ref{homoneg}
(iv);

(c) an orbit that tends to the fixed point $(r_2,0)$, corresponding to Theorem \ref{homoneg}
(v);

(d) an orbit that is ejected from the fixed point $(r_2,0)$, corresponding to Theorem \ref{homoneg} (vi);

(e) orbits that come from infinity in the direction of the stable manifold of $(r_2,0)$ and
inside it, hit the $r$ axis to the right of $r_2$, and return symmetrically to infinity in the
direction of the unstable manifold of $(r_2,0)$; these orbits correspond to
Theorem \ref{homoneg} (vii);

(f) orbits that come from infinity in the direction of the stable manifold of $(r_2,0)$ and
outside it, turn around the homoclinic loop, and return symmetrically to infinity in the
direction of the unstable manifold of $(r_2,0)$; these orbits correspond to
Theorem \ref{homoneg} (viii).

Since no other orbits show up, the proof of this case is complete.

\subsubsection{\bf The case of one fixed point} We left the case of one fixed
point at the end because it is non-generic. It occurs when the two fixed points 
of the previous case overlap. Let us denote this fixed point by $(r_0,0)$. Then 
the function $g(r)$ is positive everywhere except at the fixed point, where 
it is zero. So near $r_0$, $g$ is decreasing for $r<r_0$ and increasing for $r>r_0$,
and the $r$ axis is tangent to the graph of $g$. Consequently, 
${\partial G\over\partial r}(r_0,0)=0$, and the eigenvalues obtained
from equation \eqref{eigenv} are $\lambda_1=\lambda_2=0$. In this degenerate
case, the orbits near the fixed point influence the asymptotic behavior of the flow
at $(r_0,0)$. Since the flow away from the fixed point looks very much like in
the case of no fixed points, the only difference between the flow sketched
in Figure \ref{Fig2}(a) and the current case is that at least an orbit ends at $(r_0,0)$, 
and at least another orbit one ejects from it. These orbits are described in Theorem \ref{homoneg} (iv) and (v).

\medskip

The proof of Theorem \ref{homoneg} is now complete.


\section{Mass equality of Lagrangian solutions}

In this section we show that all Lagrangian solutions that satisfy 
Definition \ref{deflag} must have equal masses. In other words,
we will prove the following result:

\begin{theorem}
In the curved $3$-body problem, the bodies of masses $m_1, m_2$, $m_3$ can
lead to a Lagrangian solution if and only if $m_1=m_2=m_3$.
\end{theorem}
\begin{proof} 
The fact that three bodies of equal masses can lead to Lagrangian solutions
for suitable initial conditions was proved in Theorem \ref{equal-masses}. So 
we will further prove that Lagrangian solutions can occur only if the masses are
equal. Since the case of relative equilibria was settled in \cite{Diacu1},
we need to consider only the Lagrangian orbits that are not relative equilibria. 
This means we can treat only the case when $r(t)$ is not constant.

Assume now that the masses are $m_1, m_2, m_3$, and substitute a
solution of the form
\begin{align*}
x_1&=r\cos\omega,& y_1&=r\sin\omega,& z_1&=(\sigma\kappa^{-1}-\sigma r^2)^{1/2},\\
x_2&=r\cos(\omega +2\pi/3),& y_2&=r\sin(\omega +2\pi/3),& z_2&=(\sigma\kappa^{-1}-\sigma r^2)^{1/2},\\
x_3&=r\cos(\omega +4\pi/3),& y_3&=r\sin(\omega +4\pi/3),& z_3&=(\sigma\kappa^{-1}-\sigma r^2)^{1/2},
\end{align*}
into the equations of motion. Computations and a reasoning similar to the ones 
performed in the proof of Theorem \ref{equal-masses} lead us to the system:
$$\ddot r-r(1-\kappa r^2)\dot\omega^2+{\kappa r\dot r^2\over{1-\kappa r^2}}+
{12(m_1+m_2)(1-\kappa r^2)\over{r^2(12-9\kappa r^2)^{3/2}}}=0,$$
$$\ddot r-r(1-\kappa r^2)\dot\omega^2+{\kappa r\dot r^2\over{1-\kappa r^2}}+
{12(m_2+m_3)(1-\kappa r^2)\over{r^2(12-9\kappa r^2)^{3/2}}}=0,$$
$$\ddot r-r(1-\kappa r^2)\dot\omega^2+{\kappa r\dot r^2\over{1-\kappa r^2}}+
{12(m_3+m_1)(1-\kappa r^2)\over{r^2(12-9\kappa r^2)^{3/2}}}=0,$$
$$r\ddot\omega+2\dot r\dot\omega-{4\sqrt{3}(m_1-m_2)\over r^2(12-9\kappa r^2)^{3/2}}=0,$$
$$r\ddot\omega+2\dot r\dot\omega-{4\sqrt{3}(m_2-m_3)\over r^2(12-9\kappa r^2)^{3/2}}=0,$$
$$r\ddot\omega+2\dot r\dot\omega-{4\sqrt{3}(m_3-m_1)\over r^2(12-9\kappa r^2)^{3/2}}=0,$$
which, obviously, can have solutions only if $m_1=m_2=m_3$. This conclusion completes
the proof.
\end{proof}


\section{Local existence and uniqueness of Eulerian solutions}

In this section we define the Eulerian solutions of the curved 3-body
problem and prove their local existence for suitable initial conditions in
the case of equal masses.

\begin{definition}
A solution of equations \eqref{second} is called Eulerian if, at every time instant,
the bodies are on a geodesic that contains the point $(0,0, |\kappa|^{-1/2}|)$.
\label{defeul}
\end{definition}

According to Definition \ref{defeul}, the size of an Eulerian solution may change,
but the particles are always on a (possibly rotating) geodesic. If the masses 
are equal, it is natural to assume that one body lies at the point 
$(0,0, |\kappa|^{-1/2}|)$, while the other two bodies find 
themselves at diametrically opposed points of a circle. Thus, in the case 
of equal masses, which we further consider, we ask that the moving bodies 
have the same coordinate $z$, which may vary in time.  

We can thus represent such an Eulerian solution of the curved 3-body problem in the form
\begin{equation}
{\bf q}
=({\bf q}_1,{\bf q}_2, {\bf q}_3), \ \ {\rm with}\ \ {\bf q}_i=(x_i,y_i,z_i),\  i=1,2,3,
\label{eulsolu}
\end{equation}
\begin{align*}
x_1&=0,& y_1&=0,& z_1&=(\sigma\kappa)^{-1/2},\\
x_2&=r\cos\omega,& y_2&=r\sin\omega,& z_2&=z,\\
x_3&=-r\cos\omega,& y_3&=-r\sin\omega,& z_3&=z,
\end{align*}
where $z=z(t)$ satisfies $z^2=\sigma\kappa^{-1}-\sigma r^2=(\sigma\kappa)^{-1}
(1-\kappa r^2)$; $\sigma$ is the signature function defined in \eqref{sigma}; $r:=r(t)$ is the {\it size function}; and $\omega:=\omega(t)$ is the {\it angular function}.  

Notice that, for every time $t$, we have $x_i^2(t)+y_i^2(t)+\sigma z_i^2(t)=\kappa^{-1},\ i=1,2,3$, which means that the bodies stay on the surface ${\bf M}_{\kappa}^2$. Equations
\eqref{eulsolu} also express the fact that the bodies are on the same (possibly rotating) geodesic. Therefore representation \eqref{eulsolu} of the Eulerian orbits agrees with Definition \ref{defeul} in the case of equal masses.

\begin{definition}
An Eulerian solution of equations \eqref{second} is called Eulerian homothetic if
the configuration expands or contracts, but does not rotate. 
\label{ellipticeulhomo}
\end{definition}

In terms of representation \eqref{eulsolu}, an Eulerian homothetic 
orbit for equal masses occurs when $\omega(t)$ is constant, but $r(t)$ is not constant. 
If, for instance, all three bodies
are initially in the same open hemisphere, while the two moving 
bodies have the same mass and the same $z$ coordinate, and are released
with zero initial velocities, then we are led to an Eulerian homothetic orbit
that ends in a triple collision.

\begin{definition}
An Eulerian solution of equations \eqref{second} is called an Eulerian relative 
equilibrium if the configuration of the system rotates without expanding or
contracting.
\end{definition}

In terms of representation \eqref{eulsolu}, an Eulerian relative
equilibrium orbit occurs when $r(t)$ is constant, but $\omega(t)$ is not constant. Of course, Eulerian homothetic solutions and elliptic Eulerian relative equilibria, whose existence we proved in \cite{Diacu1}, are particular Eulerian orbits, but we expect that the Eulerian orbits are not reduced to them. We  now show this fact by proving the local existence and uniqueness of Eulerian solutions that are neither Eulerian homothetic, nor Eulerian relative equilibria. 


\begin{theorem}
In the curved $3$-body problem of equal masses, for every set of initial conditions 
belonging to a certain class, the local existence and uniqueness of an Eulerian solution,
which is neither homothetic nor a relative equilibrium, is assured. \label{Eulequal-masses}
\end{theorem}
\begin{proof}
To check whether equations \eqref{second} admit solutions of the form \eqref{eulsolu}
that start in the region $z>0$ and for which both $r(t)$ and $\omega(t)$ are not constant, we first compute that 
$$\kappa{\bf q}_1\odot{\bf q}_2=\kappa{\bf q}_1\odot{\bf q}_3=(1-\kappa r^2)^{1/2},$$
$$\kappa{\bf q}_2\odot{\bf q}_3=1-2\kappa r^2,$$
\begin{align*}
\dot x_1&=0,& \dot y_1&=0,\\
\dot x_2&=\dot r\cos\omega-r\dot\omega\sin\omega,& \dot y_2&=\dot r\sin\omega+
r\dot\omega\cos\omega,\\
\dot x_3&=-\dot r\cos\omega+r\dot\omega\sin\omega,& \dot y_2&=-\dot r\sin\omega-
r\dot\omega\cos\omega,\\
\dot z_1&=0,& \dot z_2=\dot z_3&=-{\sigma r\dot r\over (\sigma\kappa)^{1/2}(1-\kappa r^2)^{1/2}},
\end{align*}
$$\kappa\dot{\bf q}_1\odot\dot{\bf q}_1=0,$$
$$\kappa\dot{\bf q}_2\odot\dot{\bf q}_2=\kappa\dot{\bf q}_3\odot\dot{\bf q}_3=\kappa r^2\dot\omega^2+{\kappa\dot r^2\over 1-\kappa r^2},$$
$$\ddot x_1=\ddot y_1=\ddot z_1=0,$$
$$\ddot x_2=(\ddot r-r\dot\omega^2)\cos\omega-(r\ddot\omega+2\dot r\dot\omega)\sin\omega,$$
$$\ddot y_2=(\ddot r-r\dot\omega^2)\sin\omega+(r\ddot\omega+2\dot r\dot\omega)\cos\omega,$$
$$\ddot x_3=-(\ddot r-r\dot\omega^2)\cos\omega+(r\ddot\omega+2\dot r\dot\omega)\sin\omega,$$
$$\ddot y_3=-(\ddot r-r\dot\omega^2)\sin\omega-(r\ddot\omega+2\dot r\dot\omega)\cos\omega,$$
$$\ddot z_2=\ddot z_3=-\sigma r\ddot r(\sigma\kappa^{-1}-\sigma r^2)^{-1/2}-
\kappa^{-1}\dot r^2(\sigma\kappa^{-1}-\sigma r^2)^{-3/2}.$$
Substituting these expressions into equations \eqref{second}, we are led to the
system below, where the double-dot terms on the left indicate to which differential
equation each algebraic equation corresponds:
\begin{align*}
\ddot x_2, \ddot x_3:\ \ \ \ \ \ \ \ & C\cos\omega-D\sin\omega=0,\\
\ddot y_2, \ddot y_3:\ \ \ \ \ \ \ \ & C\sin\omega+D\cos\omega=0,\\
\ddot z_2, \ddot z_3:\ \ \ \ \ \ \ \ & C=0,
\end{align*}
where
$$C:=C(t)=\ddot r-r(1-\kappa r^2)\dot\omega^2+{\kappa r\dot r^2\over{1-\kappa r^2}}+
{m(5-4\kappa r^2)\over{4r^2(1-\kappa r^2)^{1/2}}},$$
$$D:=D(t)=r\ddot\omega+2\dot r\dot\omega.$$
(The equations corresponding to $\ddot x_1, \ddot y_1,$ and $\ddot z_1$ are identities, so
they don't show up). The above system has solutions if and only if $C=D=0$, which means that the existence of Eulerian homographic orbits of the curved $3$-body problem with equal masses is equivalent to the existence of solutions of the system of
differential equations:
\begin{equation}
\begin{cases}
\dot r=\nu\cr
\dot w=-{2\nu w\over r}\cr
\dot\nu=r(1-\kappa r^2)w^2-{\kappa r\nu^2\over{1-\kappa r^2}}-
{m(5-4\kappa r^2)\over{4r^2(1-\kappa r^2)^{1/2}}}, \cr
\end{cases}\label{eu0}
\end{equation}
with initial conditions $r(0)=r_0, w(0)=w_0, \nu(0)=\nu_0,$ where $w=\dot\omega$.
The functions $r,\omega$, and $w$ are analytic, and as long as the initial
conditions satisfy the conditions $r_0>0$ for all $\kappa$, as well as $r_0<\kappa^{-1/2}$
for $\kappa>0$, standard results of the theory of differential equations guarantee the
local existence and uniqueness of a solution $(r,w,\nu)$ of equations \eqref{eu0},
and therefore the local existence and uniqueness of an Eulerian orbit with
$r(t)$ and $\omega(t)$ not constant. This conclusion completes the proof. 
\end{proof}


\section{Classification of Eulerian solutions for $\kappa>0$}

We can now state and prove the following result:

\begin{theorem}
In the curved $3$-body problem with equal masses and $\kappa>0$ there are 
three classes of Eulerian solutions:

(i) homothetic orbits that begin or end in total collision in finite time;

(ii) relative equilibria, for which one mass is fixed at one pole of
the sphere while the other two move on a circle parallel with the $xy$ plane;

(iii) periodic homographic orbits that are not relative equilibria.

None of the above orbits can cross the equator, defined as the great circle orthogonal
to the $z$ axis.
\label{eulerp}
\end{theorem}

The rest of this section is dedicated to the proof of this theorem.

Let us start by noticing that the first two equations of system \eqref{eu0} 
imply that $\dot w=-{2\dot r w\over r}$, which leads to 
$$w=\frac{c}{r^2},$$
where $c$ is a constant. The case $c=0$ can occur only when $w=0$,
which means $\dot\omega=0$. Under these circumstances the angular
velocity is zero, so the motion is homothetic. The existence of these orbits 
is stated in Theorem \ref{eulerp} (i). They occur 
only when the angular momentum is zero, and lead to a triple collision
in the future or in the past, depending on the direction of the velocity vectors.
The existence of the orbits described in Theorem \ref{eulerp} (ii) was proved in \cite{Diacu1}.

For the rest of this section, we assume that $c\ne 0$. System \eqref{eu0} is 
thus reduced to
\begin{equation}
\begin{cases}
\dot r=\nu\cr
\dot\nu={c^2(1-\kappa r^2)\over r^3}-{\kappa r\nu^2\over{1-\kappa r^2}}-
{m(5-4\kappa r^2)\over{4r^2(1-\kappa r^2)^{1/2}}}. \cr
\end{cases}\label{eu}
\end{equation}
To address the existence of the orbits described in Theorem \ref{eulerp} (iii), 
and show that no other Eulerian orbits than those of Theorem \ref{eulerp}
exist for $\kappa>0$, we need to study the 
flow of system \eqref{eu} for $\kappa>0$. Let us first prove the following fact:

\begin{lemma}
Regardless of the values of the parameters $m, \kappa>0$, and $c\ne 0$, 
system \eqref{eu} has one fixed point $(r_0,0)$ with $0<r_0<\kappa^{-1/2}$.
\label{lemmaeup}
\end{lemma}
\begin{proof} 
The fixed points of system \eqref{eu} are of the form $(r,0)$ for all values of $r$ 
that are zeroes of $u(r)$, where
\begin{equation}
u(r)={c^2(1-\kappa r^2)\over r}-
{m(5-4\kappa r^2)\over{4(1-\kappa r^2)^{1/2}}}.\label{uu}
\end{equation} 
But finding the zeroes of $u(r)$ is equivalent to obtaining the roots of the polynomial
$$
16\kappa^2(c^4\kappa + m^2)r^6 - 8\kappa(6c^4\kappa + 5m^2)r^4 +
(48c^4\kappa+25m^2)r^2 - 16c^4.
$$
Denoting $x=r^2$, this polynomial becomes
$$
q(x)=16\kappa^2(c^4\kappa + m^2)x^3 - 8\kappa(6c^4\kappa + 5m^2)x^2
+ (48c^4\kappa+25m^2)x - 16c^4.
$$
Since $\kappa>0$, Descarte's rule of signs implies that
$q$ can have one or three positive roots. The derivative of $q$ is the polynomial
\begin{equation}\label{derEuler-pol2}
q'(x)=48\kappa^2(c^4\kappa + m^2)x^2 - 16\kappa(6c^4\kappa + 5m^2)x
+ 48c^4\kappa+25m^2,
\end{equation}
whose discriminant is $64\kappa^2m^2(21c^4\kappa + 25m^2)$.
But, as $\kappa>0$, this discriminant is always positive, so it offers no 
additional information on the total number of positive roots. 

To determine the exact number of positive roots, we will use the resultant 
of two polynomials. Denoting by $a_i, i=1,2,\dots, \zeta,$ the roots of a polynomial $P$,
and by $b_j, j=1,2, \dots, \xi$, those of a polynomial $Q$, the resultant of $P$ and 
$Q$ is defined by the expression
$${\rm Res}(P,Q)=\prod _{i=1}^\zeta\prod_{j=1}^\xi(a_i-b_j).$$
Then $P$ and $Q$ have a common root if and only if 
${\rm Res}[P,Q]=0$. Consequently the resultant of $q$ and $q'$ 
is a polynomial in $\kappa, c$, and $m$ whose zeroes are the double roots of 
$q$. But
$$
{\rm Res}(q,q') = 1024c^4\kappa^5m^4(c^4\kappa + m^2)(108c^4\kappa +
125m^2).
$$
Then, for $m,\kappa >0$ and $c\ne 0$, ${\rm Res}[q,q']$ never cancels, 
therefore $q$ has exactly one positive root. Indeed, should $q$ have
three positive roots, a continuous variation of $\kappa, m$, and $c$, would lead
to some values of the parameters that correspond to a double root. Since
double roots are impossible, the existence of a unique equilibrium $(r_0,0)$ 
with $r_0>0$ is proved. To conclude that $r_0<\kappa^{-1/2}$ for all 
$m,\kappa>0$ and $c\ne 0$, it is enough to notice that
$\lim_{r\to 0}u(r)  = +\infty$ and $\lim_{r\to \kappa^{-1/2}}u(r)= - \infty.$
This conclusion completes the proof.
\end{proof}

\subsection{The flow in the $(r,\nu)$ plane for $\kappa>0$}

We can now study the flow of system \eqref{eu} in the $(r,\nu)$
plane for $\kappa>0$. The vector field is not defined along the 
lines $r=0$ and $r=\kappa^{-1/2}$, so it lies in the band 
$(0,\kappa^{-1/2})\times \mathbb{R}$. Consider now the slope ${d\nu\over dr}$
of the vector field. This slope is given by the ratio ${{\dot\nu}\over{\dot r}}=v(r,\nu)$, where
\begin{equation}
v(r,\nu)={c^2(1-\kappa r^2)\over \nu r^3}-{\kappa r\nu\over{1-\kappa r^2}}-
{m(5-4\kappa r^2)\over{4\nu r^2(1-\kappa r^2)^{1/2}}}.
\label{slope2}
\end{equation}
But $v$ is odd with respect to $\nu$, i.e.\ $v(r,-\nu)=-v(r,\nu)$, so the vector
field is symmetric with respect to the $r$ axis.  

\begin{figure}[htbp] 
   \centering
   \includegraphics[width=2in]{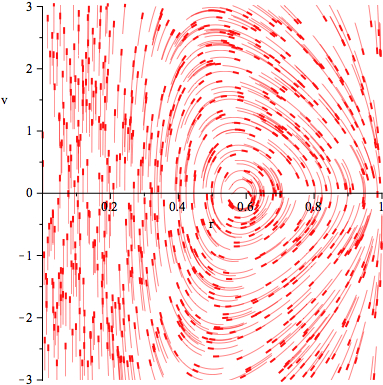}
   \caption{A sketch of the flow of system \eqref{eu} for 
   $\kappa = 1, c = 2,$ and  $m = 2$, typical for Eulerian
   solutions with $\kappa>0$.}
   \label{Fig3}
\end{figure}


Since $\lim_{r\to 0}v(r) = +\infty$ and $\lim_{r\to \kappa^{-1/2}}v(r) = - \infty$, 
the flow crosses the $r$ axis perpendicularly upwards to the left of $r_0$ and 
downwards to its right, where $(r_0,0)$ is the fixed point of the system \eqref{eu}
whose existence and uniqueness we proved in Lemma \ref{lemmaeup}. 
But the right hand side of the second equation in \eqref{eu} is of the
form 
\begin{equation}
W(r,\nu)=u(r)/r^2+g_2(r,\nu),\label{W}
\end{equation} 
where $g_2$ was defined earlier as 
$g_2(r,\nu)=-{{\kappa r\nu^2}
\over{1-\kappa r^2}}$, while $u(r)$ was defined in \eqref{uu}. 
Notice that
$$\lim_{r\to 0}W(r,\nu)=+\infty\ \ {\rm and}\ \ \lim_{r\to\kappa^{-1/2}}W(r,\nu)=-\infty.$$
Moreover, $W(r_0,0)=0$, and $W$ has no singularities for $r\in(0,\kappa^{-1/2})$.
Therefore the flow that enters the region $\nu>0$ to the left of $r_0$ must exit
it to the right of the fixed point. The symmetry with respect to the
$r$ axis forces all orbits to be periodic around $(r_0,0)$ (see Figure \ref{Fig3}). 
This proves the existence of the solutions described in Theorem \ref{eulerp} (iii),
and shows that no orbits other than those in Theorem \ref{eulerp} occur
for $\kappa>0$. The proof of Theorem \ref{eulerp} is now complete.


\section{Classification of Eulerian solutions for $\kappa<0$}

We can now state and prove the following result:

\begin{theorem}
In the curved $3$-body problem with equal masses and $\kappa>0$ there are 
four classes of Eulerian solutions:

(i) Eulerian homothetic orbits that begin or end in total collision in finite time;

(ii) Eulerian relative equilibria, for which one mass is fixed at the vertex of the hyperboloid while the other two move on a circle parallel with the $xy$ plane; 

(iii) Eulerian periodic orbits that are not relative equilibria; the line connecting the
two moving bodies is always parallel with the $xy$ plane, but their $z$ coordinate
changes in time;

(iv) Eulerian orbits that come from infinity at time $-\infty$, reach a position
when the size of the configuration is minimal, and then return to infinity at time
$+\infty$.

\label{eulneg}
\end{theorem}

The rest of this section is dedicated to the proof of this theorem.

The homothetic orbits of the type stated in Theorem \ref{eulneg} (i) occur
only when $c=0$. Then the two moving bodies collide simultaneously with
the fixed one in the future or in the past. Depending on the initial
conditions, the motion can be bounded or unbounded.

The existence of the orbits stated in Theorem \ref{eulneg} (ii) was proved in
\cite{Diacu1}. To prove the existence of the solutions stated in Theorem \ref{eulneg} (iii)
and (iv), and show that there are no other kinds of orbits, we start with the following
result:

\begin{lemma}
In the curved three body problem with $\kappa<0$, the polynomial $q$ 
defined in the proof of Lemma \ref{lemmaeup} has no positive roots for 
$c^4\kappa + m^2 \leq 0$, but has exactly one positive root for 
$c^4\kappa + m^2 > 0$.
\label{last}
\end{lemma}
\begin{proof}
We split our analysis in three different cases depending on the sign of $c^4\kappa + m^2$:

(1) $c^4\kappa + m^2=0$.
In this case $q$ has form
$8\kappa m^2 x^2 + 23c^4\kappa x -16c^4,$ a polynomial that does not have any
positive root.

(2) $c^4\kappa + m^2<0.$
Writing $6c^4\kappa + 5m^2 = 6(c^4\kappa + m^2) - m^2$, we see that
the term of $q$ corresponding to $x^2$ is always negative, so by Descartes's 
rule of signs the number of positive roots depends on the sign of the coefficient
corresponding to $x$, i.e.\ $48c^4\kappa+25m^2=48(c^4\kappa+m^2)-23m^2$, which
is also negative, and therefore $q$ has no positive root.

(3) $c^4\kappa + m^2>0$.
This case leads to three subcases:

-- if $6c^4\kappa + 5m^2<0$, then necessarily $48c^4\kappa+25m^2<0$
and, so $q$ has exactly one positive root;

-- if $6c^4\kappa + 5m^2>0$ and $48c^4\kappa+25m^2<0$, then $q$ has 
one change of sign and therefore exactly one positive root;

-- if $48c^4\kappa+25m^2>0,$ then all coefficients, except for the
free term, are positive, therefore $q$ has exactly one positive root. 

These conclusions complete the proof. \end{proof}

The following result will be used towards understanding the case when 
system \eqref{eu} has one fixed point.

\begin{lemma}
Regardless of the values of the parameters $\kappa<0, m>0$, and 
$c\ne 0$, there is no fixed point, $(r_*,0)$, of system \eqref{eu} for which
${\partial\over\partial r}W(r_*,0)=0$, where $W$ is defined in \eqref{W}.
\label{lllast}
\end{lemma}
\begin{proof}
Since $u(r_*)=0, {\partial\over \partial r}g_2(r_*,0)=0$, and
$${\partial W\over \partial r}(r,\nu)=-(2/r^3)u(r)+(1/r^2){d\over dr}u(r)+{\partial\over \partial r}
g_2(r,\nu),$$
it means that $W(r_*,0)=0$ if and only if ${d\over dr}u(r_*)=0$. Consequently our result would follow
if we can prove that there is no fixed point $(r_*,0)$ for which ${d\over dr}u(r_*)=0$.
To show this fact, notice first that
\begin{equation}
{d\over dr}u(r)=-{{c^2(1+\kappa r^2)}\over{r^2}}-{{\kappa mr(4\kappa r^2-3)}
\over{4(1-\kappa r^2)^{3/2}}}.
\label{derivu}
\end{equation}
From the definition of $u(r)$ in \eqref{uu}, the identity $u(r_*)=0$ is equivalent to
$$(1-\kappa r_*^2)^{1/2}={mr_*(5-4\kappa r_*^2)\over 4c^2(1-\kappa r_*^2)}.$$
Regarding $(1-\kappa r^2)^{3/2}$ as $(1-\kappa r^2)^{1/2}(1-\kappa r^2)$,
and substituting the above expression of $(1-\kappa r_*^2)^{1/2}$ into
\eqref{derivu} for $r=r_*$, we obtain that
$${\kappa(4\kappa r_*^2-3)\over 5-4\kappa r_*^2}+{{1+\kappa r_*^2}\over r_*^2}=0,$$
which leads to the conclusion that $r_*^2=5/2\kappa<0$. Therefore there is no
fixed point $(r_*,0)$ such that ${d\over dr}u(r_*)=0$. This conclusion completes the proof.
\end{proof}

\subsection{The flow in the $(r,\nu)$ plane for $\kappa<0$}
To study the flow of system \eqref{eu} for $\kappa<0$, we will consider
the two cases given by Lemma \ref{last}, namely when system \eqref{eu}
has no fixed points and when it has exactly one fixed point.

\subsubsection{\bf The case of no fixed points} Since $\kappa<0$, and 
system \eqref{eu} has no fixed points, the function $u(r)$, defined in \eqref{uu}, 
has no zeroes. But $\lim_{r\to 0}u(r)=+\infty$, so $u(r)>0$ for all $r>0$. Then 
$u(r)/r^2>0$ for all $r>0$. Since $\lim_{r\to 0}g_2(r,\nu)=0$,
it follows that $\lim_{r\to 0}W(r,\nu)=+\infty$, where $W(r,\nu)$ (defined in \eqref{W}) 
forms the right hand side of the second equation in system \eqref{eu}. Since system
\eqref{eu} has no fixed points, $W$ doesn't vanish. Therefore $W(r,\nu)>0$ for all 
$r>0$ and $\nu$.

Notice that the slope of the vector field, $v(r,\nu)$, defined in \eqref{slope2}, is
of the form $v(r,\nu)=W(r,\nu)/\nu$, which implies that the flow crosses the
$r$ axis perpendicularly at every point with $r>0$. Also, for $r$ fixed, 
$\lim_{\nu\to\pm\infty}W(r,\nu)=+\infty$.
Moreover, for $\nu$ fixed, $\lim_{r\to\infty}W(r,\nu)=0$. This means that
the flow has a simple behavior as in Figure \ref{Fig4}(a). These orbits correspond
to those stated in Theorem \ref{eulneg} (iv).

\begin{figure}[htbp] 
   \centering
   \includegraphics[width=2in]{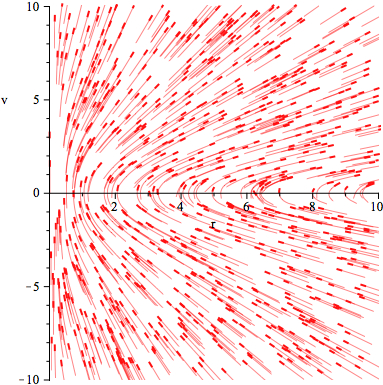}
   \includegraphics[width=2in]{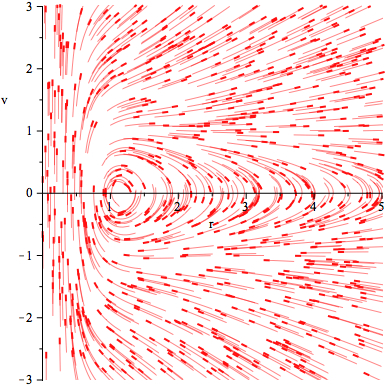}
   \caption{A sketch of the flow of system \eqref{eu} for (a) 
   $\kappa = -2, c = 2,$ and  $m = 4$, typical for no fixed points;
   (b) $\kappa = -2, c = 2,$ and  $m = 6.2$, typical for one fixed point.}
   \label{Fig4}
\end{figure}


\subsubsection{\bf The case of one fixed point}
We start with analyzing the behavior of the flow near the unique
fixed point $(r_0,0)$. Let $F(r,\nu)=\nu$ denote the right hand side in the
first equation of system \eqref{eu}. Then ${\partial\over\partial r}F(r_0,0)=0,
{\partial\over\partial\nu}F(r_0,0)=1$, and ${\partial\over\partial\nu}W(r_0,0)=0$.
To determine the sign of ${\partial\over\partial r}W(r_0,0)$, notice first that
$\lim_{r\to 0}W(r,\nu)=+\infty$. Since the equation $W(r,\nu)=0$ has a single
root of the form $(r_0,0)$, with $r_0>0$, it follows that  $W(r,0)>0$ for
$0<r<r_0$.

To show that $W(r,0)<0$ for $r>r_0$, assume the contrary, which (given
the fact that $r_0$ is the only zero of $W(r,0)$) means that $W(r,0)>0$ for
$r>r_0$. So $W(r,0)\ge 0$, with equality only for $r=r_0$. But recall that 
we are in the case when the parameters satisfy the inequality 
$c^4\kappa + m^2 > 0$. Then a slight variation of the parameters 
$\kappa<0, m>0$, and $c\ne 0$, within the region defined by the above inequality,
leads to two zeroes for $W(r,0)$, a fact which contradicts Lemma \ref{last}. Therefore, necessarily, $W(r,0)<0$ for $r>r_0$. 

Consequently $W(r,0)$ is decreasing in a small neighborhood of $r_0$, so ${\partial\over\partial r}W(r_0,0)\le 0$. But by Lemma \ref{lllast}, ${\partial\over\partial r}W(r_0,0)\ne 0$,
so necessarily ${\partial\over\partial r}W(r_0,0)< 0$. The eigenvalues 
corresponding to the system obtained by linearizing equations \eqref{eu} around the fixed point $(r_0,0)$ are given by the equation
\begin{equation}
\det\begin{bmatrix}
-\lambda & 1\\
{\partial W\over\partial r}(r_0,0) & -\lambda
\end{bmatrix}
=0,\label{eigenvalues}
\end{equation}
so these eigenvalues are purely imaginary. In terms of system \eqref{eu}, this means
that the fixed point $(r_0,0)$ could be a spiral sink, a spiral source, or a center. The
symmetry of the flow with respect to the $r$ axis excludes the first two possibilities,
consequently $(r_0,0)$ is a center (see Figure \ref{Fig4}(b)). We thus proved that, in a neighborhood of this fixed point, there exist infinitely many periodic Eulerian solutions whose existence was stated in Theorem \ref{eulneg} (iii). 

To complete the analysis of the flow of system \eqref{eu}, we will use the
nullcline $\dot\nu=0$, which is given by the equation
\begin{equation}
\nu^2={{1-\kappa r^2}\over \kappa r} \bigg[{c^2(1-\kappa r^2)\over r^3}-{m(5-4\kappa^2)\over4r^2(1-\kappa r^2)^{1/2}}\bigg].
\label{curve}
\end{equation}
Along this curve, which passes through the fixed point $(r_0,0)$, and is symmetric with 
respect to the $r$ axis, the vector field has slope zero. To understand the qualitative
behavior of this curve, notice that 
$$\lim_{r\to\infty} {{1-\kappa r^2}\over \kappa r} \bigg[{c^2(1-\kappa r^2)\over r^3}-{m(5-4\kappa^2)\over4r^2(1-\kappa r^2)^{1/2}}\bigg]=m(-\kappa)^{1/2}+\kappa c^2.$$
But we are restricted to the parameter region given by the inequality $m^2+\kappa c^4>0$,
which is equivalent to  
$$[m(-\kappa)^{1/2}-(-\kappa)c^2][m(-\kappa)^{1/2}+(-\kappa)c^2]>0.$$
Since the second factor of this product is positive, it follows that the first factor
must be positive, therefore the above limit is positive. Consequently the curve
given in \eqref{curve} is bounded by the horizontal lines 
$$\nu=[m(-\kappa)^{1/2}+\kappa c^2]^{1/2}\ \ {\rm and}\ \ \nu=-[m(-\kappa)^{1/2}+\kappa c^2]^{1/2}.$$
Inside the curve, the vector field has negative slope for $\nu>0$ and positive slope for $\nu<0$. Outside the curve, the vector field has positive slope for $\nu>0$, but negative
slope for $\nu<0$. So the orbits of the flow that stay outside the nullcline curve are
unbounded. They correspond to solutions whose existence was stated
in Theorem \ref{eulneg} (iv). This conclusion completes the proof of Theorem \ref{eulneg}.


\section{Hyperbolic homographic solutions}

In this last section we consider a certain class of homographic orbits, which occur only
in spaces of negative curvature. In the case $\kappa=-1$, we proved in \cite{Diacu1}
the existence of hyperbolic Eulerian relative equilibria of the curved 3-body problem
with equal masses. These orbits behave as follows: three bodies of equal masses move along three fixed hyperbolas, each body on one of them; the middle hyperbola, which is a geodesic passing through the vertex of the hyperboloid, lies in a plane of ${\mathbb R}^3$  that is parallel and equidistant from the planes containing the other two hyperbolas,
none of which is a geodesic. At every moment in time, the bodies lie equidistantly from 
each other on a geodesic hyperbola that rotates hyperbolically. These solutions are the hyperbolic counterpart of Eulerian solutions, in the sense that the bodies stay on the same geodesic, which rotates hyperbolically, instead of circularly. The existence proof we gave in \cite{Diacu1} works for any $\kappa<0$. We therefore provide the following definitions.

\begin{definition}
A solution of the curved $3$-body problem is called hyperbolic homographic
if the bodies maintain a configuration similar to itself while rotating hyperbolically.
When the bodies remain on the same hyperbolically rotating geodesic,
the solution is called hyperbolic Eulerian. 
\label{defhyp}
\end{definition}

While there is, so far, no evidence of hyperbolic non-Eulerian
homographic solutions, we showed in \cite{Diacu1} that hyperbolic Eulerian
orbits exist in the case of equal masses. In the particular case of equal masses, 
it is natural to assume that the middle body moves on a geodesic passing through the
point $(0,0,|\kappa|^{-1/2})$ (the vertex of the hyperboloid's upper sheet),
while the other two bodies are on the same (hyperbolically rotating) geodesic,
and equidistant from it. 

Consequently we can seek hyperbolic Eulerian solutions of equal masses of the form:
\begin{equation}
{\bf q}=({\bf q}_1,{\bf q}_2, {\bf q}_3),\  {\rm with}\  {\bf q}_i=(x_i,y_i,z_i),\  i=1,2,3,\  
{\rm and}\label{hypsolu}
\end{equation}
\begin{align*}
x_1&=0,& y_1&=|\kappa|^{-1/2}\sinh\omega,& z_1&=|\kappa|^{-1/2}\cosh\omega,\\
x_2&=(\rho^2+\kappa^{-1})^{1/2},& y_2&=\rho\sinh\omega,& z_2&=\rho\cosh\omega,\\
x_3&=-(\rho^2+\kappa^{-1})^{1/2},& y_3&=\rho\sinh\omega,& z_3&=\rho\cosh\omega,
\end{align*}
where $\rho:=\rho(t)$ is the {\it size function} and $\omega:=\omega(t)$ is the
{\it angular function}. 

Indeed, for every time $t$, we have that $x_i^2(t)+y_i^2(t)-z_i^2(t)=\kappa^{-1},\ i=1,2,3$, which means that the bodies stay on the surface ${\bf H}_{\kappa}^2$, while lying on the same, possibly (hyperbolically) rotating, geodesic. Therefore representation \eqref{hypsolu} of the hypebolic Eulerian homographic orbits agrees with Definition \ref{defhyp}.

With the help of this analytic representation, we can define Eulerian homothetic orbits and hyperbolic Eulerian relative equilibria.

\begin{definition}
A hyperbolic Eulerian homographic solution is called Eulerian homothetic if the 
configuration of the system expands or contracts, but does not rotate hyperbolically.
\label{hypeulhomo}
\end{definition}

In terms of representation \eqref{hypsolu}, an Eulerian homothetic solution occurs
when $\omega(t)$ is constant, but $\rho(t)$ is not constant. A straightforward computation shows that if $\omega(t)$ is constant, the bodies lie initially on a geodesic, 
and the initial velocities are such that the bodies move along the geodesic 
towards or away from a triple collision at the point occupied by the fixed body.

Notice that Definition \ref{hypeulhomo} 
leads to the same orbits produced by Definition \ref{ellipticeulhomo}. While
the configuration of the former solution does not rotate hyperbolically, and
the configuration of the latter solution does not rotate elliptically, both fail to rotate while expanding or contracting. This is the reason why Definitions  \ref{ellipticeulhomo} and \ref{hypeulhomo} use the same name (Eulerian homothetic) for these types of orbits.

\begin{definition}
A hyperbolic Eulerian homographic solution is called a hyperbolic Eulerian relative equilibrium if the configuration rotates hyperbolically while its size remains constant.
\end{definition}

In terms of representation \eqref{hypsolu}, hyperbolic Eulerian relative equilibria occur
when $\rho(t)$ is constant, while $\omega(t)$ is not constant.

Unlike for Lagrangian and Eulerian solutions, hyperbolic Eulerian homographic
orbits exist only in the form of homothetic solutions or relative equilibria. As we will
further prove, any composition between a homothetic orbit and a relative equilibrium
fails to be a solution of system \eqref{second}.

\begin{theorem}
In the curved $3$-body problem of equal masses with $\kappa<0$, the only hyperbolic Eulerian homographic solutions are either Eulerian homothetic orbits or hyperbolic
Eulerian relative equilibria.
\end{theorem}
\begin{proof}
Consider for system \eqref{second} a solution of the form \eqref{hypsolu} that is
not homothetic. Then
$$\kappa{\bf q}_1\odot{\bf q}_2=\kappa{\bf q}_1\odot{\bf q}_3=|\kappa|^{1/2}\rho,$$
$$\kappa{\bf q}_2\odot{\bf q}_3=-1-2\kappa \rho^2,$$
\begin{align*}
{\dot x}_1&={\ddot x}_1=0,& {\dot y}_1&=|\kappa|^{-1/2}\dot\omega\cosh\omega,& {\dot z}_1&=|\kappa|^{-1/2}\sinh\omega,
\end{align*}
$${\dot x}_2=-{\dot x}_3={\rho\dot\rho\over{(\rho^2+\kappa^{-1})^{1/2}}},$$ 
$${\dot y}_2={\dot y}_3=\dot\rho\sinh\omega+\rho\dot\omega\cosh\omega,$$
$${\dot z}_2={\dot z}_3=\dot\rho\cosh\omega+\rho\dot\omega\sinh\omega,$$
$$\kappa\dot{\bf q}_2\odot\dot{\bf q}_2=-\dot\omega^2,\ \ \ \kappa\dot{\bf q}_2\odot\dot{\bf q}_2=
\kappa\dot{\bf q}_3\odot\dot{\bf q}_3=\kappa\rho^2\dot\omega^2-{\kappa\dot\rho^2\over{1+\kappa\rho^2}},$$
$${\ddot x}_2=-{\ddot x}_3={\rho\ddot\rho\over{(\rho^2+\kappa^{-1})^{1/2}}}+
{\kappa^{-1}\dot\rho^2\over{(\rho^2+\kappa^{-1})^{3/2}}},$$ 
$${\ddot y}_2={\ddot y}_3=(\ddot\rho+\rho\dot\omega^2)\sinh\omega+
(\rho\ddot\omega+2\dot\rho\dot\omega)\cosh\omega,$$
$${\ddot z}_2={\ddot z}_3=(\ddot\rho+\rho\dot\omega^2)\cosh\omega+
(\rho\ddot\omega+2\dot\rho\dot\omega)\sinh\omega.$$
Substituting these expressions in system \eqref{second}, we are led to an identity
corresponding to $\ddot x_1$. The other equations lead to the system
\begin{align*}
\ddot x_2, \ddot x_3:\ \ \ \ \ \ \ \ & E=0\\
\ddot y_1:\ \ \ \ \ \ \ \ & |\kappa|^{-1/2}\ddot\omega\cosh\omega=0,\\ 
\ddot z_1:\ \ \ \ \ \ \ \ & |\kappa|^{-1/2}\ddot\omega\sinh\omega=0,\\ 
\ddot y_2, \ddot y_3:\ \ \ \ \ \ \ \ & E\sinh\omega+F\cosh\omega=0,\\
\ddot z_2, \ddot z_3:\ \ \ \ \ \ \ \ & E\cosh\omega+F\sinh\omega=0,\\
\end{align*}
where
$$E:=E(t)=\ddot\rho+\rho(1+\kappa\rho^2)\dot\omega^2-{\kappa\rho\dot\rho^2\over{1+\kappa\rho^2}}+{m(1-4\kappa\rho^2)\over{4\rho^2|1+\kappa\rho^2|^{1/2}}},$$
$$F:=F(t)=\rho\ddot\omega+2\dot\rho\dot\omega.$$
This system can be obviously satisfied only if
\begin{equation}
\begin{cases}
\ddot\omega=0\cr
\ddot\omega=-{2\dot\rho\dot\omega\over\rho}\cr
\ddot\rho=-\rho(1+\kappa\rho^2)\dot\omega^2+{\kappa\rho\dot\rho^2\over{1+\kappa\rho^2}}-{m(1+4\kappa\rho^2)\over{4\rho^2|1+\kappa\rho^2|^{1/2}}}.\cr
\end{cases}
\end{equation}
The first equation implies that $\omega(t)=at+b$, where $a$ and $b$ are constants,
which means that $\dot\omega(t)=a$. Since we assumed that the solution is not
homothetic, we necessarily have $a\ne 0$. But from the second equation, we can
conclude that
$$\dot\omega(t)={c\over \rho^2(t)},$$
where $c$ is a constant. Since $a\ne 0$, it follows that $\rho(t)$ is constant,
which means that the homographic solution is a relative equilibrium. This conclusion
is also satisfied by the third equation, which reduces to
$$a^2=\dot\omega^2={m(1-4\kappa\rho)\over{4\rho^3|1+\kappa\rho^2|^{1/2}}},$$
being verified by two values of $a$ (equal in absolute value, but of opposite signs) for every $\kappa$ and $\rho$ fixed.
Therefore every hyperbolic Eulerian homographic solution that is not Eulerian homothetic
is a hyperbolic Eulerian relative equilibrium. This conclusion completes the proof.
\end{proof}

Since a slight perturbation of hyperbolic Eulerian relative equilibria, within the
set of  hyperbolic Eulerian homographic solutions, produces no orbits with 
variable size, it means that hyperbolic Eulerian relative equilibria of equal masses 
are unstable. So though they exist in a mathematical sense, as proved above (as
well as in \cite{Diacu1}, using a direct method), such equal-mass orbits are unlikely to be found in a (hypothetical) hyperbolic physical universe.


\bigskip

\noindent{\bf Acknowledgments.} Florin Diacu was supported by an NSERC Discovery Grant, while Ernesto P\'erez-Chavela enjoyed the financial support of CONACYT.


\end{document}